\documentclass{amsart}
\usepackage{amssymb}
\usepackage{amsmath}
\usepackage{amsfonts}

\newtheorem{theorem}{Theorem}
\theoremstyle{plain}

\newtheorem{corollary}{Corollary}

\newtheorem{definition}{Definition}

\newtheorem{lemma}{Lemma}

\newtheorem{proposition}{Proposition}
\newtheorem{remark}{Remark}

\numberwithin{equation}{section}

\begin{document}
\title[Conjugate Connections]{Conjugate Connections and their Applications
on Pure Metallic Metric Geometries}
\author{Olgun DURMAZ}
\address{Ataturk University, Faculty of Science, Department of Mathematics,
25240, Erzurum-Turkey.}
\email{olgun.durmaz@atauni.edu.tr}
\author{Aydin GEZER}
\address{Ataturk University, Faculty of Science, Department of Mathematics,
25240, Erzurum-Turkey.}
\email{aydingzr@gmail.com}

\begin{abstract}
Let $\left( M,J,g\right) $ be a metallic pseudo-Riemannian manifold equipped
with a metallic structure $J$ and a pseudo-Riemannian metric $g$. The paper
deals with interactions of Codazzi couplings formed by conjugate connections
and tensor structures. The presence of Tachibana operator and Codazzi
couplings presented a new characterization for locally metallic
pseudo-Riemannian manifold. Also, a necessary and sufficient condition a
non-integrable metallic pseudo-Riemannian manifold is a quasi metallic
pseudo Riemannian manifold is derived. Finally, it is introduced
metallic-like pseudo-Riemannian manifolds and presented some results
concerning them.


\end{abstract}

\maketitle

\section{\protect\bigskip Introduction}

The concept of tensor structures on differentiable manifolds is one of the
central concepts of modern differential geometry. Among tensor structures on
a differentiable manifold, one of the best known is almost complex
structure, that is, a $(1,1)-$ tensor field $J$ whose square, at each point,
is minus the identity. The manifold must be even-dimensional, that is, $%
dim=2n$. Usually, it is equipped with a Hermitian metric $g$ which is a
(Riemannian or pseudo-Riemannian) metric that preserves the almost complex
structure, that is, the almost complex structure $J$ acts as an isometry
with respect to the (pseudo-)Riemannian metric. In this case, the $(J,g)$ is
called an almost Hermitian structure. The associated $(0,2)-$tensor of the
Hermitian metric $g$ is a $2-$form $\omega $ which is defined by $\omega
=g(JX,Y)$ and hence the relationship with symplectic geometry.

The other case is that the almost complex structure $J$ acts as an
anti-isometry with respect to a pseudo-Riemannian metric $g$. Such a metric
is known as an anti-Hermitian metric or a Norden metric (first studied by
and named after Norden \cite{Norden}). The Norden metric is necessary
pseudo-Riemannian of neutral signature. In this case, $(J,g)$ is called an
almost anti-Hermitian structure or almost Norden structure or almost $B-$%
structure. The associated $(0,2)-$tensor of any Norden metric which is
defined by $G=g(JX,Y)$ is also a Norden metric. So, in this case we dispose
with a pair of mutually associated Norden metrics, known also as twin Norden
metrics.

Besides almost complex structures, other tensor structures on differentiable
manifolds which are called almost tangent, almost product structures etc.
exist and are important for modern differential geometry. All the tensor
structures we have talked about so far are actually a polynomial structure.
Polynomial structures can be considered as $(1,1)-$ tensor field which are
roots of the algebraic equation%
\begin{equation*}
Q(J)=J^{n}+a_{n}J^{n-1}+...+a_{2}J+a_{1}J=0,
\end{equation*}%
where $a_{1},a_{2},...,a_{n}$ are real numbers and $I$ is the identity
tensor of type $(1,1)$. Polynomial structures on a manifold were defined in 
\cite{Goldberg}. In particular, we can say the followings

\begin{itemize}
\item If $Q(J)=J^{2}+I=0$, its solution $J$ is called an almost complex
structure.

\item If $Q(J)=J^{2}-I=0$, its solution $J$ is called an almost product
structure.

\item If $Q(J)=J^{2}=0$, its solution $J$ is called an almost tangent
structure.

\item If $Q(J)=J^{2}-J-I=0$, its solution $J$ is called an golden structure.
This tensor structure was inspired by the golden ratio, which was described
by Johannes Kepler (1571--1630). The number $\eta =\frac{1+\sqrt{5}}{2}%
=1,618...$, which is a solution of the equation $x^{2}-x-1=0$, is the golden
ratio.

\item If $Q(J)=J^{2}-pJ-qI=0$, , with $p$ and $q$ positive integers, its
solution $J$ is called a metallic structure. The name is motivated by the
fact that the metallic means family or metallic proportions introduced by by
de Spinadel \cite{Spinadel} is the positive root of the quadratic equation $%
x^{2}-px-q=0$, namely $\sigma _{p,q}=\frac{p+\sqrt{p^{2}+4q}}{2}$. If $p=0$,
that is, $J^{2}=qI$, then we call them trivial metallic structures.

We will follow the notations and definitions concerning with metallic
pseudo-Riemannian manifolds given in \cite{Blaga1}.
\end{itemize}

\begin{definition}
\cite{Blaga1} Let $M$ be a smooth (real) manifold and $J:TM\rightarrow TM$
be a tangent bundle isomorphism. A $(1,1)-$tensor field $J$ on $M$ is called
a metallic structure if $J^{2}=pJ+qI$, where $I$ is the identity operator
and $p,q\in 
\mathbb{R}
.$
\end{definition}

The classical definition is used for integrability of a metallic structure.

\begin{definition}
\cite{Blaga1} For a metallic structure $J$, the integrability is equivalent
to the vanishing of the Nijenhuis tensor $N_{J}$ : 
\begin{equation*}
N_{J}(X,Y)=\left[ JX,JY\right] -J\left[ JX,Y\right] -J\left[ X,JY\right]
+J^{2}\left[ X,Y\right] .
\end{equation*}%
Here and further, $X,Y,Z$ will stand for arbitrary differentiable vector
fields on the considered manifold (or vectors in its tangent space at an
arbitrary point of the manifold).
\end{definition}

\begin{definition}
A linear connection $\nabla $ on a smooth manifold $M$ is called a $J-$%
connection if $\nabla J=0$.
\end{definition}

Another equivalent definition for integrability of a polynomial structure is
given by Vanzura in \cite{Vanzura}: in order that a metallic structure $J$
be integrable, it is necessary and sufficient that we can introduce a
torsion-free linear connection $\nabla $ such that $\nabla J=0$, that is,
the torsion-free linear connection $\nabla $ is a $J-$connection.

\begin{definition}
Let $\left( M,g\right) $ be a pseudo-Riemannian manifold and $J$ be a $g-$%
symmetric $\left( 1,1\right) -$tensor field on $M$ such that $J^{2}=pJ+qI$,
where $p$ and $q$ are real numbers. Then $\left( J,g\right) $ is called a
metallic pseudo-Riemannian structure and $\left( M,J,g\right) $ is named a
metallic pseudo-Riemannian manifold. A metallic pseudo-Riemannian manifold $%
\left( M,J,g\right) $ such that the Levi-Civita connection $\nabla ^{g}$
with respect to $g$ is a $J-$connection is called a locally metallic
pseudo-Riemannian manifold.
\end{definition}

Next, we can give the following proposition for another way of describing
locally metallic pseudo-Riemannian manifolds. The following proposition can
be proven by following the method used in \cite{Salimov1}. It will be
presented without proof. So, we will avoid unnecessary repetition.

\begin{proposition}
\label{Proposition7} Let $\left( M,J,g\right) $ be a metallic
pseudo-Riemannian manifold and $\nabla ^{g}$ be the Levi-Civita connection
of $g$. Then $\Phi _{J}g=0$ is equivalent to $\nabla ^{g}J=0$, where $\Phi
_{J}$ is the Tachibana operator defined by 
\begin{equation*}
\left( \Phi _{J}g\right) \left( X,Y,Z\right) =JXg\left( Y,Z\right) -X\left(
g\left( JY,Z\right) \right) +g\left( \left( L_{Y}J\right) X,Z\right)
+g\left( Y,\left( L_{Z}J\right) X\right) .
\end{equation*}%
Here $\left( L_{X}J\right) Y=\left[ X,JY\right] -J\left[ X,Y\right] .$
\end{proposition}

Now let's turn our attention to the other important actor in our article,
the conjugate connections. The notion of conjugate connections with respect
to a metric tensor field was originally introduced by Norden in the context
of Weyl geometry \cite{Norden}. Such linear connections were independently
developed by Nagaoka and Amari \cite{Nagaoka} under the name dual
connections and used by Lauritzen in the definition of statistical manifolds 
\cite{Lauritzen}. In this direction, we can define two conjugate connection
on a metallic pseudo-Riemannian manifold $\left( M,J,g\right) $, which are
respectively called $g-$conjugate connection ($\nabla ^{\ast }$) and $G-$%
conjugate connection ($\nabla ^{\dag }$):%
\begin{equation}
Zg\left( X,Y\right) =g\left( \nabla _{Z}X,Y\right) +g\left( X,\nabla
_{Z}^{\ast }Y\right)  \label{GD1}
\end{equation}%
and%
\begin{equation}
ZG\left( X,Y\right) =G\left( \nabla _{Z}X,Y\right) +G\left( X,\nabla
_{Z}^{\dag }Y\right) ,  \label{GD2}
\end{equation}%
where $G=g(JX,Y)$ is the twin metric on $\left( M,J,g\right) $ and $\nabla $
is a linear connection. It is easy to see that $\nabla ^{\ast }$ and $\nabla
^{\dag }$ are indeed connections such that $\left( \nabla ^{\ast }\right)
^{\ast }=\left( \nabla ^{\dag }\right) ^{\dag }=\nabla $. Another type of
conjugate connections are those concerning with tensor structures (for
details, \cite{Alek,Bejan,Blaga2}). In our setting, this connection is
defined by 
\begin{equation*}
\nabla _{X}^{J}Y=J^{-1}\left( \nabla _{X}\left( JY\right) \right)
\end{equation*}%
which is called $J-$conjugate connection ($\nabla ^{J}$).

\section{Codazzi Coupling of $\protect\nabla $ with $J$}

Let $\nabla $ be an arbitrary linear connection on a pseudo-Riemannian
manifold $(M,g)$. A symmetric $(0,2)-$tensor field $\rho $ is Codazzi if it
satisfies the symmetry property%
\begin{equation*}
(\nabla _{X}\rho )(Y,Z)=(\nabla _{Z}\rho )(X,Y).
\end{equation*}%
Alternatively, a $(1,1)-$tensor field $J$ is Codazzi if it is self-adjoint
and%
\begin{equation*}
(\nabla _{X}J)Y=(\nabla _{Y}J)X.
\end{equation*}%
We call the pairs $(\nabla ,\rho )$ and $(\nabla ,J)$, respectively, a
Codazzi-coupled.

Let $\left( M,J,g\right) $ be a metallic pseudo-Riemannian manifold. The
inverse of the metallic structure $J$ is as follow: $J^{-1}=\frac{1}{q}%
(J-pI) $ $\left( q\neq 0\right) $. The following proposition is analogue to
the known result given by Fei and Zhang \cite{Fei} for Hermitian setting.

\begin{proposition}
Let $\nabla $ be a linear connection and $J$ be a metallic structure on $M$.
Then the following situations are equivalent

$1)\ \left( \nabla ,J\right) $ is Codazzi-coupled;

$2)\ \nabla $ and $\nabla ^{J}$ have equal torsions;

$3)\ \left( \nabla ^{J},J^{-1}\right) $ is Codazzi-coupled.
\end{proposition}

\begin{proof}
$1)\Rightarrow 2):\ $Assume that $\left( \nabla ,J\right) $ is
Codazzi-coupled. Then we have%
\begin{eqnarray}
&&J^{-1}\left( \left( \nabla _{X}J\right) Y-\left( \nabla _{Y}J\right)
X\right)  \label{GD10} \\
&=&J^{-1}\left( \nabla _{X}JY-J\nabla _{X}Y-\nabla _{Y}JX+J\nabla
_{Y}X\right)  \notag \\
&=&J^{-1}\left( \nabla _{X}JY\right) -J^{-1}\left( \nabla _{Y}JX\right)
-\left( \nabla _{X}Y-\nabla _{Y}X\right)  \notag \\
&=&\nabla _{X}^{J}Y-\nabla _{Y}^{J}X-\left( \nabla _{X}Y-\nabla _{Y}X\right)
\notag \\
&=&T^{\nabla ^{J}}\left( X,Y\right) -T^{\nabla }\left( X,Y\right) .  \notag
\end{eqnarray}%
$2)\Rightarrow 3):\ $Suppose that we have $T^{\nabla ^{J}}\left( X,Y\right)
=T^{\nabla }\left( X,Y\right) $. Since%
\begin{eqnarray}
&&J\left( \left( \nabla _{X}^{J}J^{-1}\right) Y-\left( \nabla
_{Y}^{J}J^{-1}\right) X\right)  \label{GD11} \\
&=&J\left( \nabla _{X}^{J}J^{-1}Y\right) -\nabla _{X}^{J}Y-J\left( \nabla
_{Y}^{J}J^{-1}X\right) +\nabla _{Y}^{J}X  \notag \\
&=&\nabla _{X}Y-\nabla _{Y}X-\left( \nabla _{X}^{J}Y-\nabla _{Y}^{J}X\right)
\notag \\
&=&T^{\nabla }\left( X,Y\right) -T^{\nabla ^{J}}\left( X,Y\right) ,  \notag
\end{eqnarray}%
from which it is clear that $\left( \nabla ^{J},J^{-1}\right) $ is
Codazzi-coupled.

$3)\Rightarrow 1):$ The result immediately follows from (\ref{GD10}) and (%
\ref{GD11}).
\end{proof}

Fei and Zhang \cite{Fei} proved that if $J$ is either an almost complex
structure or an almost para-complex structure, then $\left( \nabla
^{J}\right) ^{J}=\nabla $. For the metallic structure $J$, we have the
following proposition.

\begin{proposition}
\label{Proposition2}Let $J$ be a metallic structure on $M$ and $\nabla $ be
a linear connection. If $J$ is a trivial metallic structure $(p=0)$, then we
get $\left( \nabla ^{J}\right) ^{J}=\nabla $.
\end{proposition}

\begin{proof}
For a linear connection $\nabla $ and a metallic structure $J$, we know that%
\begin{equation*}
\nabla _{X}^{J}Y=J^{-1}\left( \nabla _{X}\left( JY\right) \right) .
\end{equation*}%
Using the definition of metallic structure $J$, we can rewrite it as%
\begin{eqnarray*}
\left( \nabla ^{J}\right) _{X}^{J}Y &=&J^{-1}\left( \nabla _{X}^{J}\left(
JY\right) \right) \\
&=&J^{-2}\left( \left( \nabla _{X}J^{2}\right) Y+J^{2}\nabla _{X}Y\right) \\
&=&pJ^{-2}\left( \left( \nabla _{X}J\right) Y\right) +qJ^{-2}\left( \left(
\nabla _{X}I\right) Y\right) +\nabla _{X}Y \\
&=&pJ^{-2}\left( \left( \nabla _{X}J\right) Y\right) +\nabla _{X}Y,
\end{eqnarray*}%
Thus, if $p=0,$ then we have $\left( \nabla ^{J}\right) _{X}^{J}Y=\nabla
_{X}Y.$
\end{proof}

\begin{remark}
Let $M$ be an almost (para)complex manifold equipped a $(1,1)-$tensor field $%
J$ such that $J^{2}=-I$ (or $J^{2}=I)$ where $I$ is the identity operation.
Fei and Zhang obtained a nice result stating that $g-$conjugation, $\omega -$%
conjugation and $J-$conjugation (along with an identity operation) together
form a $4-$element Klein group on the space of linear connections, where $%
\omega $ is the fundemantal $2-$form (see Theorem 2.13 in \cite{Fei}).
Unfortunately, for metallic structure $J$, $J-$conjugation is not
involutive, that is, $\left( \nabla ^{J}\right) ^{J}\neq \nabla $. So, in
our setting, we cannot create a $4-$element Klein group with our arguments.
But, this shows the difference of our structure from structures used in \cite%
{Fei}.
\end{remark}

\begin{proposition}
\label{Proposition1} Let $J$ be a metallic structure on $M$ and $\nabla $ be
a linear connection on $M$ . If $\left( \nabla ,J\right) $ is a
Codazzi-coupled, then$\ \left( \nabla ^{J},J\right) $ is so.
\end{proposition}

\begin{proof}
Assume that $\left( \nabla ,J\right) $ is a Codazzi-coupled. Then, standard
calculations give 
\begin{eqnarray*}
\left( \nabla _{X}^{J}J\right) Y-\left( \nabla _{Y}^{J}J\right) X &=&\nabla
_{X}^{J}\left( JY\right) -J\nabla _{X}^{J}Y-\nabla _{Y}^{J}\left( JX\right)
+J\nabla _{Y}^{J}X \\
&=&J^{-1}\left( \left( \nabla _{X}J^{2}\right) Y+J^{2}\left( \nabla
_{X}Y\right) \right) -\nabla _{X}\left( JY\right) \\
&&-J^{-1}\left( \left( \nabla _{Y}J^{2}\right) X+J^{2}\left( \nabla
_{Y}X\right) \right) +\nabla _{Y}\left( JX\right) \\
&=&pJ^{-1}\left( \left( \nabla _{X}J\right) Y-\left( \nabla _{Y}J\right)
X\right) \\
&&+qJ^{-1}\left( \left( \nabla _{X}I\right) Y-\left( \nabla _{Y}I\right)
X\right) \\
&=&0.
\end{eqnarray*}
\end{proof}

Let $J$ be either an almost complex structure or an almost para-complex
structure. Then $\left( \nabla ,J\right) $ is Codazzi-coupled if and only if$%
\ \left( \nabla ^{J},J\right) $ is Codazzi-coupled \cite{Fei}. For the
metallic structure $J$, the reverse of the Proposition \ref{Proposition1} is
only true for trivial metallic structure $(p=0)$, which is meaningless.

\begin{lemma}
\label{Lemma1} Let $J$ be a metallic structure on $M$ and $\nabla $ be a
linear connection on $M$. If $\left( \nabla ,J\right) $ is a
Codazzi-coupled, then 
\begin{equation*}
N_{J}\left( X,Y\right) =-J^{2}T^{\nabla }\left( X,Y\right) +T^{\nabla
}\left( X,JY\right) +JT^{\nabla }\left( JX,Y\right) -T^{\nabla }\left(
JX,JY\right) ,
\end{equation*}%
where $T^{\nabla }\left( X,Y\right) =\nabla _{X}Y-\nabla _{Y}X-\left[ X,Y%
\right] $ (also, see \cite{Fei}).
\end{lemma}

For the integrability of a metallic structure $J$, the following proposition
immediately follows from Lemma \ref{Lemma1}.

\begin{proposition}
Let $J$ be a metallic structure on $M$ and $\nabla $ be a linear connection
on $M$. If $\left( \nabla ,J\right) $ is Codazzi-coupled and $T^{\nabla
}\left( X,Y\right) =0$, then the metallic structure $J$ is integrable.
\end{proposition}

\section{Codazzi Coupling of $\protect\nabla $ with $g$ and $G$}

Given a pseudo-Riemannian manifold $(M,g)$ endowed with a metallic structure 
$J$, then the triple $(M,J,g)$ is called a metallic pseudo-Riemannian
manifold if \cite{Blaga1}

\begin{equation*}
g(JX,Y)=g(X,JY).
\end{equation*}
Also, the twin metallic pseudo-Riemannian metric $G$ is defined by

\begin{equation*}
G(X,Y)=g(JX,Y).
\end{equation*}%
From the equalities of the above, it is easy to see the following result.

\begin{proposition}
Let $(M,J,g)$ be metallic pseudo-Riemannian manifold and $G$ be the twin
pseudo-Riemannian metric. Then the following equalities hold:

$1)\ G\left( X,Y\right) =G\left( Y,X\right) $;

$2)\ G\left( JX,Y\right) =G\left( X,JY\right) $, that is, $(M,J,G)$ is a
metallic pseudo-Riemannian manifold;

$3)\ g\left( J^{-1}X,Y\right) =g\left( X,J^{-1}Y\right) $, that is, $%
(M,J^{-1},g)$ is a metallic pseudo-Riemannian manifold;

$4)\ G\left( J^{-1}X,Y\right) =G\left( X,J^{-1}Y\right) $, that is, $%
(M,J^{-1},G)$ is a metallic pseudo-Riemannian manifold.
\end{proposition}

Let $C$ and $\Gamma $ be the $\left( 0,3\right) $-tensor defined by%
\begin{equation}
C\left( X,Y,Z\right) \equiv \left( \nabla _{Z}g\right) \left( X,Y\right)
=Zg\left( X,Y\right) -g\left( \nabla _{Z}X,Y\right) -g\left( X,\nabla
_{Z}Y\right)  \label{GD3}
\end{equation}%
and 
\begin{equation}
\Gamma \left( X,Y,Z\right) \equiv \left( \nabla _{Z}G\right) \left(
X,Y\right) =ZG\left( X,Y\right) -G\left( \nabla _{Z}X,Y\right) -G\left(
X,\nabla _{Z}Y\right) .  \label{GD5}
\end{equation}%
Due to symmetry of $g$ and $G$, it is clear that $C\left( X,Y,Z\right)
=C\left( Y,X,Z\right) $ and $\Gamma \left( X,Y,Z\right) =\Gamma \left(
Y,X,Z\right) $. From (\ref{GD1}) and (\ref{GD3}) (resp., (\ref{GD2}) and (%
\ref{GD5})), we get 
\begin{equation*}
C\left( X,Y,Z\right) =g\left( X,\left( \nabla ^{\ast }-\nabla \right)
_{Z}Y\right)
\end{equation*}%
and%
\begin{equation*}
\Gamma \left( X,Y,Z\right) =G\left( X,\left( \nabla ^{\dag }-\nabla \right)
_{Z}Y\right) .
\end{equation*}%
Also, it is easy to see that%
\begin{equation}
C^{\ast }\left( X,Y,Z\right) \equiv \left( \nabla _{Z}^{\ast }g\right)
\left( X,Y\right) =-C\left( X,Y,Z\right)  \label{GD6}
\end{equation}%
and 
\begin{equation}
\Gamma ^{\dag }\left( X,Y,Z\right) \equiv \left( \nabla _{Z}^{\dag }G\right)
\left( X,Y\right) =-\Gamma \left( X,Y,Z\right) .  \label{GD7}
\end{equation}%
For the metallic structure $J$, the relationship between $C=\nabla g$ and $%
\Gamma =\nabla G$ is as follow%
\begin{equation}
\Gamma \left( X,Y,Z\right) =C\left( X,JY,Z\right) +g\left( X,\left( \nabla
_{Z}J\right) Y\right) .  \label{GD8}
\end{equation}%
From (\ref{GD8}), if $\nabla $ is a $J-$connection, then we get $\Gamma
\left( X,Y,Z\right) =C\left( X,JY,Z\right) $.

Next, we give the following proposition which is analogue to result given in 
\cite{Fei} for Hermitian setting.

\begin{proposition}
Let $G$ be the twin metallic pseudo-Riemannian metric, $\nabla $ be a linear
connection and $\nabla ^{\dag }$ be the $G-$conjugate connection of $\nabla $%
. Then the following conditions are equivalent

$1)\ \left( \nabla ,G\right) $ is Codazzi-coupled,

$2)\ \left( \nabla ^{\dag },G\right) $ is Codazzi-coupled,

$3)\ \Gamma $ is totally symmetric,

$4)\ \Gamma ^{\dag }$ is totally symmetric,

$5)\ T^{\nabla }=T^{\nabla ^{\dag }}$,\newline
where $\Gamma =\nabla G$ and $\Gamma ^{\dag }=\nabla ^{\dag }G$.
\end{proposition}

\begin{proof}
The assertion can be easily proved by following the method in \cite%
{Fei,Gezer1}.
\end{proof}

\begin{proposition}
Let $(M,J,g)$ be a metallic pseudo-Riemannian manifold, $G$ be the twin
metallic pseudo-Riemannian metric and $\nabla $ be a linear connection. If
both $\left( \nabla ,J\right) $ and $\left( \nabla ,G\right) $ are
Codazzi-coupled, then%
\begin{equation*}
G\left( \left( \nabla _{X}^{\dag }J\right) Y-\left( \nabla _{Y}^{\dag
}J\right) X,Z\right) =G\left( Y,\left( \nabla _{Z}J\right) X-\left( \nabla
_{Z}^{\dag }J\right) X\right) ,
\end{equation*}%
where $\nabla ^{\dag }$ is the $G-$conjugate connection of $\nabla $.
\end{proposition}

\begin{proof}
If both $\left( \nabla ,J\right) $ and $\left( \nabla ,G\right) $ are
Codazzi-coupled, then%
\begin{eqnarray*}
&&G\left( \left( \nabla _{X}^{\dag }J\right) Y-\left( \nabla _{Y}^{\dag
}J\right) X,Z\right) \\
&=&G\left( \nabla _{X}^{\dag }JY-J\nabla _{X}^{\dag }Y,Z\right) -G\left(
\nabla _{Y}^{\dag }JX-J\nabla _{Y}^{\dag }X,Z\right) \\
&=&XG\left( Z,JY\right) -G\left( \nabla _{X}Z,JY\right) -YG\left(
Z,JX\right) +G\left( \nabla _{Y}Z,JX\right) \\
&&+G\left( J\nabla _{Y}^{\dag }X,Z\right) -G\left( J\nabla _{X}^{\dag
}Y,Z\right) \\
&=&\Gamma \left( Z,JY,X\right) +G\left( Z,\nabla _{X}JY\right) -\Gamma
\left( Z,JX,Y\right) -G\left( \nabla _{Y}JX,Z\right) \\
&&+G\left( JZ,\nabla _{Y}^{\dag }X-\nabla _{X}^{\dag }Y\right) \\
&=&\Gamma \left( Z,JY,X\right) -\Gamma \left( Z,JX,Y\right) +G\left(
JZ,T^{\nabla }\left( X,Y\right) -T^{\nabla ^{\dag }}\left( X,Y\right) \right)
\\
&=&\Gamma \left( X,JY,Z\right) -\Gamma \left( Y,JX,Z\right) \\
&=&G\left( Y,\left( \nabla _{Z}J\right) X-\left( \nabla _{Z}^{\dag }J\right)
X\right) ,
\end{eqnarray*}%
which completes the proof.
\end{proof}

\begin{proposition}
Let $(M,J,g)$ be a metallic pseudo-Riemannian manifold, $G$ be the twin
metallic pseudo-Riemannian metric and $\nabla $ be a linear connection. If
both $\left( \nabla ,g\right) $ and $\left( \nabla ,G\right) $ are
Codazzi-coupled, then $\left( \nabla ^{\ast },J\right) $ is Codazzi-coupled,
where $\nabla ^{\ast }$ is the $g-$conjugate connection of $\nabla $.
\end{proposition}

\begin{proof}
From (\ref{GD8}), we have 
\begin{equation*}
C\left( X,JY,Z\right) =\Gamma \left( X,Y,Z\right) -g\left( X,\left( \nabla
_{Z}J\right) Y\right)
\end{equation*}%
and%
\begin{equation*}
C\left( Z,JY,X\right) =\Gamma \left( Z,Y,X\right) -g\left( Z,\left( \nabla
_{X}J\right) Y\right) .
\end{equation*}%
Since both $\left( \nabla ,g\right) $ and $\left( \nabla ,G\right) $ are
Codazzi-coupled, we can write 
\begin{eqnarray*}
0 &=&g\left( Z,\left( \nabla _{X}J\right) Y\right) -g\left( X,\left( \nabla
_{Z}J\right) Y\right) \\
&=&g\left( Y,\left( \nabla _{X}^{\ast }J\right) Z\right) -g\left( Y,\left(
\nabla _{Z}^{\ast }J\right) X\right) \\
&=&g\left( Y,\left( \nabla _{X}^{\ast }J\right) Z-\left( \nabla _{Z}^{\ast
}J\right) X\right) ,
\end{eqnarray*}%
which gives that $\left( \nabla ^{\ast },J\right) $ is Codazzi-coupled.
\end{proof}

\begin{proposition}
Let $(M,J,g)$ be a metallic pseudo-Riemannian manifold and $G$ be the twin
metallic pseudo-Riemannian metric. Then

1) If both $\left( \nabla ,J\right) $ and $\left( \nabla ^{\ast },J\right) $
are Codazzi-coupled, then we have $\nabla =\nabla ^{\ast }$, where $\nabla
^{\ast }$ is the $g-$conjugate connection of $\nabla $;

2) If both $\left( \nabla ,J\right) $ and $\left( \nabla ^{\dag },J\right) $
are Codazzi-coupled, then we have $\nabla =\nabla ^{\dag }$, where $\nabla
^{\dag }$ is the $G-$conjugate connection of $\nabla $.
\end{proposition}

\begin{proof}
Here we will prove only \textit{1)}. Similarly, \textit{2)} can be proven.
Assume that both $\left( \nabla ,J\right) $ and $\left( \nabla ^{\ast
},J\right) $ are Codazzi-coupled. From (\ref{GD1}), it is easy to see that%
\begin{equation*}
g\left( \left( \nabla _{X}J\right) Y,Z\right) )=g\left( Y,\left( \nabla
_{X}^{\ast }J\right) Z\right) .
\end{equation*}%
Thus, we can write 
\begin{eqnarray*}
0 &=&g\left( \left( \nabla _{X}J\right) Y-\left( \nabla _{Y}J\right)
X,Z\right) \\
&=&g\left( Y,\left( \nabla _{Z}^{\ast }J\right) X\right) -g\left( Y,\left(
\nabla _{Z}J\right) X\right) \\
&=&g\left( Y,\left( \nabla _{Z}^{\ast }J\right) X-\left( \nabla _{Z}J\right)
X\right)
\end{eqnarray*}%
that implies $\nabla ^{\ast }J=\nabla J$.
\end{proof}

\begin{proposition}
\label{Proposition4}Let $(M,J,g)$ be a metallic pseudo-Riemannian manifold
and $G$ be the twin metallic pseudo-Riemannian metric. If $\nabla $ is a $J-$%
connection, then both $\nabla ^{\ast }$ and $\nabla ^{\dag }$ are $J-$%
connections.
\end{proposition}

\begin{proof}
Assume that $\nabla $ is a $J-$connection. Then, we have%
\begin{eqnarray*}
g\left( Y,\nabla _{X}^{\ast }JZ-J\nabla _{X}^{\ast }Z\right) &=&g\left(
Y,\nabla _{X}^{\ast }JZ\right) -g\left( Y,J\nabla _{X}^{\ast }Z\right) \\
&=&X\left( g\left( Y,JZ\right) \right) -g\left( \nabla _{X}Y,JZ\right)
-g\left( Y,J\nabla _{X}^{\ast }Z\right) \\
&=&X\left( g\left( Y,JZ\right) \right) -g\left( \nabla _{X}Y,JZ\right)
-X\left( g\left( JY,Z\right) \right) +g\left( \nabla _{X}JY,Z\right) \\
&=&0
\end{eqnarray*}%
such that $\nabla _{X}^{\ast }JZ=J\nabla _{X}^{\ast }Z$. Similarly, using
the twin metallic pseudo-Riemannian metric $G$, it can be proven that $%
\nabla ^{\dag }$ is a $J-$connection.
\end{proof}

\begin{proposition}
\label{Proposition3}Let $(M,J,g)$ be a metallic pseudo-Riemannian manifold, $%
G$ be the twin metallic pseudo-Riemannian metric, and $\nabla ^{\ast }$ and $%
\nabla ^{\dag }$ respectively denote the $g-$conjugate and $G-$conjugate
connection of a linear connection $\nabla $. If $\nabla $ is a $J-$%
connection, then%
\begin{eqnarray}
C\left( X,JY,Z\right) &=&-C^{\ast }\left( X,JY,Z\right) =-C^{\dag }\left(
X,JY,Z\right)  \label{GD9} \\
&=&\Gamma \left( X,Y,Z\right) =-\Gamma ^{\dag }\left( X,Y,Z\right) =-\Gamma
^{\ast }\left( X,Y,Z\right) ,  \notag
\end{eqnarray}%
where $C^{\dag }=\nabla ^{\dag }g$ and $\Gamma ^{\ast }=\nabla ^{\ast }G$.
\end{proposition}

\begin{proof}
Suppose that $\nabla $ is a $J-$connection. From (\ref{GD6}), (\ref{GD7})
and (\ref{GD9}), we have 
\begin{eqnarray*}
C\left( X,JY,Z\right) &=&-C^{\ast }\left( X,JY,Z\right) \\
\Gamma \left( X,Y,Z\right) &=&-\Gamma ^{\dag }\left( X,Y,Z\right)
\end{eqnarray*}%
and%
\begin{equation*}
C\left( X,JY,Z\right) =\Gamma \left( X,Y,Z\right) .
\end{equation*}%
Moreover, since%
\begin{eqnarray*}
\Gamma ^{\ast }\left( X,Y,Z\right) &\equiv &\left( \nabla _{Z}^{\ast
}G\right) \left( X,Y\right) =ZG\left( X,Y\right) -G\left( \nabla _{Z}^{\ast
}X,Y\right) -G\left( X,\nabla _{Z}^{\ast }Y\right) \\
&=&Zg\left( JX,Y\right) -g\left( \nabla _{Z}^{\ast }X,JY\right) -g\left(
JX,\nabla _{Z}^{\ast }Y\right) \\
&=&g\left( J\nabla _{Z}Y,X\right) -g\left( X,J\nabla _{Z}^{\ast }Y\right) \\
&=&-C\left( X,JY,Z\right)
\end{eqnarray*}%
and%
\begin{eqnarray*}
C^{\dag }\left( X,JY,Z\right) &\equiv &\left( \nabla _{Z}^{\dag }g\right)
\left( X,JY\right) =Zg\left( X,JY\right) -g\left( \nabla _{Z}^{\dag
}X,JY\right) -g\left( JX,\nabla _{Z}^{\dag }Y\right) \\
&=&ZG\left( X,Y\right) -G\left( \nabla _{Z}^{\dag }X,Y\right) -G\left(
X,\nabla _{Z}^{\dag }Y\right) \\
&=&\Gamma ^{\dag }\left( X,Y,Z\right) ,
\end{eqnarray*}%
we obtain (\ref{GD9}) which ends the proof.
\end{proof}

As a direct result of Proposition \ref{Proposition3}, we have the following
proposition.

\begin{proposition}
Let $(M,J,g)$ be a metallic pseudo-Riemannian manifold, $G$ be the twin
metallic pseudo-Riemannian metric, and $\nabla ^{\ast }$ and $\nabla ^{\dag
} $ respectively denote the $g-$conjugate and $G-$conjugate connection of a
linear connection $\nabla $. Under the assumption that $\nabla $ is a $J-$%
connection, If at least one of the pairs $\left( \nabla ,g\right) ,$ $\left(
\nabla ^{\ast },g\right) ,$ $\left( \nabla ^{\dag },g\right) ,$ $\left(
\nabla ,G\right) ,$ $\left( \nabla ^{\ast },G\right) $ and $\left( \nabla
^{\dag },G\right) $ is Codazzi-coupled, then all the remaining pairs are
Codazzi-coupled.
\end{proposition}

\begin{proposition}
\label{Proposition6} Let $(M,J,g)$ be a metallic pseudo-Riemannian manifold, 
$G$ be the twin metallic pseudo-Riemannian metric, and $\nabla ^{\ast }$ and 
$\nabla ^{\dag }$ respectively denote the $g-$conjugate and $G-$conjugate
connection of a linear connection $\nabla $. If $\nabla $ is a $J-$%
connection, then%
\begin{equation*}
\begin{array}{cc}
C\left( JX,Y,Z\right) =C\left( X,JY,Z\right) , & \Gamma \left( JX,Y,Z\right)
=\Gamma \left( X,JY,Z\right) , \\ 
C^{\ast }\left( JX,Y,Z\right) =C^{\ast }\left( X,JY,Z\right) , & \Gamma
^{\ast }\left( JX,Y,Z\right) =\Gamma ^{\ast }\left( X,JY,Z\right) , \\ 
C^{\dag }\left( JX,Y,Z\right) =C^{\dag }\left( X,JY,Z\right) , & \Gamma
^{\dag }\left( JX,Y,Z\right) =\Gamma ^{\dag }\left( X,JY,Z\right) .%
\end{array}%
\end{equation*}
\end{proposition}

\begin{proof}
If $\nabla $ is a $J-$connection, then 
\begin{eqnarray*}
C\left( JX,Y,Z\right) &=&Zg\left( JX,Y\right) -g\left( \nabla
_{Z}JX,Y\right) -g\left( JX,\nabla _{Z}Y\right) \\
&=&Zg\left( X,JY\right) -g\left( \nabla _{Z}X,JY\right) -g\left( X,\nabla
_{Z}JY\right) \\
&=&C\left( X,JY,Z\right)
\end{eqnarray*}%
and%
\begin{eqnarray*}
\Gamma \left( JX,Y,Z\right) &=&ZG\left( JX,Y\right) -G\left( \nabla
_{Z}JX,Y\right) -G\left( JX,\nabla _{Z}Y\right) \\
&=&ZG\left( X,JY\right) -G\left( \nabla _{Z}X,JY\right) -G\left( X,\nabla
_{Z}JY\right) \\
&=&\Gamma \left( X,JY,Z\right) .
\end{eqnarray*}%
With help of Proposition \ref{Proposition4}, the other equalities can be
obtained in a similar way.
\end{proof}

Next, let us introduce a tensor operator which is applied to pure tensor
fields. We refer to \cite{Salimov2} for details about pure tensor fields and
tensor operators. One of the most important classes of tensor operators is
the class of Tachibana operators associated with a considering fixed $(1,1)-$%
tensor field. The Tachibana operators were firstly defined and studied by
Tachibana in \cite{Tachiban}. Lather, the Tachibana operators was applied in
Norden geometry and in theory of lifts in \cite%
{Salimov1,Salimov2,Salimov3,Salimov4,Salimov5}. In here, we aims to give a
new characterization of locally metallic pseudo-Riemannian manifolds by
means of the Tachibana operator.

\begin{proposition}
\label{Proposition5}Let $(M,J,g)$ be a metallic pseudo-Riemannian manifold
and $\nabla $ be a torsion-free connection. Then 
\begin{equation*}
\left( \Phi _{J}g\right) \left( X,Y,Z\right) =C\left( Y,Z,JX\right) -\Gamma
\left( Y,Z,X\right) +g\left( \left( \nabla _{Y}J\right) X,Z\right) +g\left(
\left( \nabla _{Z}J\right) X,Y\right) ,
\end{equation*}%
where $\Phi _{J}g:T_{2}^{0}\left( M\right) \rightarrow T_{3}^{0}\left(
M\right) $ is the Tachibana operator applied to the pseudo-Riemannian metric 
$g$.
\end{proposition}

\begin{proof}
Using the definition of the operator $\Phi _{J}g$, we have%
\begin{equation*}
\left( \Phi _{J}g\right) \left( X,Y,Z\right) =JXg\left( Y,Z\right) -X\left(
g\left( JY,Z\right) \right) +g\left( \left( L_{Y}J\right) X,Z\right)
+g\left( Y,\left( L_{Z}J\right) X\right) ,
\end{equation*}%
where $\left( L_{X}J\right) Y=\left[ X,JY\right] -J\left[ X,Y\right] $.
Since $\nabla $ is a torsion-free linear connection, we can rewrite it as%
\begin{eqnarray*}
\left( \Phi _{J}g\right) \left( X,Y,Z\right) &=&JXg\left( Y,Z\right)
-X\left( g\left( JY,Z\right) \right) \\
&&+g\left( \left( \nabla _{Y}J\right) X,Z\right) -g\left( \nabla
_{JX}Y,Z\right) +g\left( J\nabla _{X}Y,Z\right) \\
&&+g\left( Y,\left( \nabla _{Z}J\right) X\right) -g\left( Y,\nabla
_{JX}Z\right) +g\left( Y,J\nabla _{X}Z\right) \\
&=&C\left( Y,Z,JX\right) -\Gamma \left( Y,Z,X\right) +g\left( \left( \nabla
_{Y}J\right) X,Z\right) +g\left( \left( \nabla _{Z}J\right) X,Y\right) ,
\end{eqnarray*}%
where $\left[ X,Y\right] =\nabla _{X}Y-\nabla _{Y}X$.
\end{proof}

A statistical manifold is a (pseudo-) Riemannian manifold $(M,g,\nabla )$
with a (pseudo-) Riemannian metric $g$ and a torsion-free linear connection $%
\nabla $ for which $C=\nabla g$ is totally symmetric, that is, the Codazzi
equation holds \cite{Lauritzen}%
\begin{equation*}
(\nabla _{X}g)(Y,Z)=(\nabla _{Y}g)(X,Z)=(\nabla _{Z}g)(X,Y).
\end{equation*}%
The following theorem gives a new characterization of locally metallic
pseudo-Riemannian manifolds by means of the Tachibana operator and
Codazzi-coupled.

\begin{theorem}
Let $(M,J,g)$ be a metallic pseudo-Riemannian manifold and $\nabla $ be a
torsion-free connection. If $\nabla $ is a $J-$connection and $\left( \nabla
,g\right) $ is Codazzi-coupled, then $\Phi _{J}g=0$. Furthermore, the triple 
$(M,J,g)$ is a locally metallic pseudo-Riemannian manifold.
\end{theorem}

\begin{proof}
From Proposition \ref{Proposition5} and $\nabla J=0$, we have%
\begin{equation*}
\left( \Phi _{J}g\right) \left( X,Y,Z\right) =C\left( Y,Z,JX\right) -C\left(
Y,JZ,X\right) .
\end{equation*}%
Since $\left( \nabla ,g\right) $ is Codazzi-coupled, from Proposition \ref%
{Proposition6}, we get the following%
\begin{eqnarray*}
\left( \Phi _{J}g\right) \left( X,Y,Z\right) &=&C\left( JX,Z,Y\right)
-C\left( X,JZ,Y\right) \\
&=&0.
\end{eqnarray*}%
Because of Proposition \ref{Proposition7}, we can say that the triple $%
(M,J,g)$ is a locally metallic pseudo-Riemannian manifold.
\end{proof}

In \cite{Salimov6}, for a non-integrable almost paracomplex manifold $(M,P)$
(that is, $P^{2}=I$, where $I=identity$) with compatible metric $g$, authors
give the following 
\begin{eqnarray*}
&&\left( \Phi _{P}g\right) \left( X,Y,Z\right) +\left( \Phi _{P}g\right)
\left( Y,Z,X\right) +\left( \Phi _{P}g\right) \left( Z,X,Y\right) \\
&=&g\left( \left( \nabla _{X}^{g}P\right) Y,Z\right) +g\left( \left( \nabla
_{Y}^{g}P\right) Z,X\right) +g\left( \left( \nabla _{Z}^{g}P\right)
X,Y\right) ,
\end{eqnarray*}%
where $\nabla ^{g}$ is the Levi-Civita connection of $g$. If $\left( \Phi
_{P}g\right) \left( X,Y,Z\right) +\left( \Phi _{P}g\right) \left(
Y,Z,X\right) +\left( \Phi _{P}g\right) \left( Z,X,Y\right) =g\left( \left(
\nabla _{X}^{g}P\right) Y,Z\right) +g\left( \left( \nabla _{Y}^{g}P\right)
Z,X\right) +g\left( \left( \nabla _{Z}^{g}P\right) X,Y\right) =0$, they
called the triple $(M,P,g)$ a quasi (para-)K\"{a}hler manifold which is
analogue of quasi K\"{a}hler manifold in \cite{Manev}.

Analogously, if a non-integrable metallic pseudo Riemannian manifold $%
(M,J,g) $ satisfies%
\begin{equation*}
\left( \Phi _{J}g\right) \left( X,Y,Z\right) +\left( \Phi _{J}g\right)
\left( Y,Z,X\right) +\left( \Phi _{J}g\right) \left( Z,X,Y\right) =0,
\end{equation*}%
then we are calling the triple $(M,J,g)$ a quasi metallic pseudo-Riemannian
manifold.

\begin{proposition}
Let $(M,J,g)$ be a non-integrable metallic pseudo-Riemannian manifold and $%
\nabla $ be a torsion-free connection. If $\left( \nabla ,g\right) $ is
Codazzi-coupled, then 
\begin{eqnarray*}
&&\left( \Phi _{J}g\right) \left( X,Y,Z\right) +\left( \Phi _{J}g\right)
\left( Y,Z,X\right) +\left( \Phi _{J}g\right) \left( Z,X,Y\right) \\
&=&g\left( \left( \nabla _{X}J\right) Y,Z\right) +g\left( \left( \nabla
_{Y}J\right) Z,X\right) +g\left( \left( \nabla _{Z}J\right) X,Y\right) .
\end{eqnarray*}
\end{proposition}

\begin{proof}
From Proposition \ref{Proposition5}, we know that 
\begin{equation*}
\left( \Phi _{J}g\right) \left( X,Y,Z\right) =C\left( Y,Z,JX\right) -\Gamma
\left( Y,Z,X\right) +g\left( \left( \nabla _{Y}J\right) X,Z\right) +g\left(
\left( \nabla _{Z}J\right) X,Y\right) .
\end{equation*}%
If $\left( \nabla ,g\right) $ is Codazzi-coupled, then 
\begin{eqnarray*}
&&\left( \Phi _{J}g\right) \left( X,Y,Z\right) +\left( \Phi _{J}g\right)
\left( Y,Z,X\right) +\left( \Phi _{J}g\right) \left( Z,X,Y\right) \\
&=&C\left( JX,Z,Y\right) +C\left( JY,X,Z\right) +C\left( JZ,Y,X\right) \\
&&-\left( \Gamma \left( Y,Z,X\right) +\Gamma \left( Z,X,Y\right) +\Gamma
\left( X,Y,Z\right) \right) \\
&&+g\left( \left( \nabla _{Y}J\right) X,Z\right) +g\left( \left( \nabla
_{Z}J\right) X,Y\right) +g\left( \left( \nabla _{Z}J\right) Y,X\right)
+g\left( \left( \nabla _{X}J\right) Y,Z\right) \\
&&+g\left( \left( \nabla _{X}J\right) Z,Y\right) +g\left( \left( \nabla
_{Y}J\right) Z,X\right)
\end{eqnarray*}%
from which we can write the following 
\begin{eqnarray*}
&&\left( \Phi _{J}g\right) \left( X,Y,Z\right) +\left( \Phi _{J}g\right)
\left( Y,Z,X\right) +\left( \Phi _{J}g\right) \left( Z,X,Y\right) \\
&=&g\left( \left( \nabla _{X}J\right) Y,Z\right) +g\left( \left( \nabla
_{Y}J\right) Z,X\right) +g\left( \left( \nabla _{Z}J\right) X,Y\right) .
\end{eqnarray*}
\end{proof}

By means of the above proposition, we get the following result.

\begin{corollary}
Let $(M,J,g)$ be a non-integrable metallic pseudo-Riemannian manifold and $%
\nabla $ be a torsion-free connection. Under the assumption that $\left(
\nabla ,g\right) $ is Codazzi-coupled, the triple $(M,J,g)$ is a quasi
metallic pseudo Riemannian manifold if and only if $g\left( \left( \nabla
_{X}J\right) Y,Z\right) +g\left( \left( \nabla _{Y}J\right) Z,X\right)
+g\left( \left( \nabla _{Z}J\right) X,Y\right) =0$.
\end{corollary}

\section{Generalized conjugate connections and their applications}

Generalizations of conjugate connections are studied in this section.
Similar problems studied for conjugate connections in the previous section
will be searched for generalized connections.

\begin{definition}
Let $(M,J,g)$ be a metallic pseudo-Riemannian manifold and $\nabla $ be a
linear connection on $M$. The generalized conjugate connection $\nabla
^{\ast \prime }$ of $\nabla $ with respect to $g$ by a $1-$form $\tau $ is
defined by \cite{Matsuzoe,Nomizu,Zhang}%
\begin{equation*}
Xg\left( Y,Z\right) =g\left( \nabla _{X}Y,Z\right) +g\left( Y,\nabla
_{X}^{\ast \prime }Z\right) -\tau \left( X\right) g\left( Y,Z\right) .
\end{equation*}
\end{definition}

It can easily be checked that the generalized conjugate connection $\nabla
^{\ast \prime }$ of $\nabla $ with respect to $g$ is involutive, that is, $%
\left( \nabla ^{\ast \prime }\right) ^{\ast \prime }=\nabla $.

\begin{proposition}
Let $(M,J,g)$ be a metallic pseudo-Riemannian manifold and $\nabla $ be a
linear connection on $M$. If $\nabla ^{\ast \prime }$is a generalized
conjugate connection of $\nabla $ with respect to $g$ by $\tau $, then we
have%
\begin{equation*}
\nabla \text{ is a }J-\text{connection if and only if }\nabla ^{\ast \prime }%
\text{ is a }J-\text{connection.}
\end{equation*}
\end{proposition}

\begin{proof}
Suppose that $\nabla $ is a $J-$connection. Using the definition of the
generalized conjugate connection, it is possible to write the following%
\begin{eqnarray*}
g\left( Y,\nabla _{X}^{\ast \prime }JZ\right) &=&Xg\left( Y,JZ\right)
-g\left( \nabla _{X}Y,JZ\right) +\tau \left( X\right) g\left( Y,JZ\right) \\
&=&g\left( \nabla _{X}JY,Z\right) +g\left( JY,\nabla _{X}^{\ast \prime
}Z\right) -g\left( \nabla _{X}Y,JZ\right) \\
&=&g\left( Y,J\nabla _{X}^{\ast \prime }Z\right)
\end{eqnarray*}%
such that $\nabla _{X}^{\ast \prime }JZ=J\nabla _{X}^{\ast \prime }Z$.
Conversely, let $\nabla ^{\ast \prime }$ be a $J-$connection. Then%
\begin{eqnarray*}
g\left( Y,\nabla _{X}JZ\right) &=&Xg\left( Y,JZ\right) -g\left( JZ,\nabla
_{X}^{\ast \prime }Y\right) +\tau \left( X\right) g\left( Y,JZ\right) \\
&=&g\left( \nabla _{X}^{\ast \prime }JY,Z\right) +g\left( JY,\nabla
_{X}Z\right) -g\left( JZ,\nabla _{X}^{\ast \prime }Y\right) \\
&=&g\left( Y,J\nabla _{X}Z\right)
\end{eqnarray*}%
which implies $\nabla _{X}JZ=J\nabla _{X}Z$.
\end{proof}

\begin{proposition}
Let $(M,J,g)$ be a metallic pseudo-Riemannian manifold and $\nabla $ be a
linear connection, and $\nabla ^{\ast \prime }$ and $\nabla ^{\dag \prime }$
denote generalized conjugate connections of $\nabla $ with respect to $g$
and $G$ by $\tau $, respectively. Then%
\begin{equation*}
\text{if }\nabla \text{ is a }J-\text{connection, then }\nabla ^{\ast \prime
}=\nabla ^{\dag \prime }\text{. }
\end{equation*}%
where $\left( \nabla ^{\dag \prime }\right) ^{\dag \prime }=\nabla $.
\end{proposition}

\begin{proof}
The definition of the generalized conjugate connection of $\nabla $ with
respect to $G$ by $\tau $ immediately gives%
\begin{eqnarray*}
XG\left( Y,Z\right) &=&G\left( \nabla _{X}Y,Z\right) +G\left( Y,\nabla
_{X}^{\dag \prime }Z\right) -\tau \left( X\right) G\left( Y,Z\right) \\
&=&g\left( J\nabla _{X}Y,Z\right) +g\left( JY,\nabla _{X}^{\dag \prime
}Z\right) -\tau \left( X\right) g\left( JY,Z\right) \\
&=&Xg\left( JY,Z\right) .
\end{eqnarray*}%
Since 
\begin{equation*}
Xg\left( JY,Z\right) =g\left( \nabla _{X}JY,Z\right) +g\left( JY,\nabla
_{X}^{\ast \prime }Z\right) -\tau \left( X\right) g\left( JY,Z\right) ,
\end{equation*}%
we find%
\begin{equation*}
g\left( \nabla _{X}JY-J\nabla _{X}Y,Z\right) +g\left( JY,\nabla _{X}^{\ast
\prime }Z-\nabla _{X}^{\dag \prime }Z\right) =0
\end{equation*}%
from which we immediately obtain that if $\nabla J=0$, then $\nabla ^{\ast
\prime }=\nabla ^{\dag \prime }$.
\end{proof}

\begin{proposition}
Let $(M,J,g)$ be a metallic pseudo-Riemannian manifold and $\nabla ^{\ast
\prime }$ be a generalized conjugate connection of $\nabla $ with respect to 
$g$ by $\tau $. Then $\nabla ^{\ast \prime }$ is a generalized conjugate
connection of $\nabla ^{J}$ with respect to $G$ by $\tau $.
\end{proposition}

\begin{proof}
We have%
\begin{eqnarray*}
G\left( \nabla _{X}^{J}Y,Z\right) &=&G\left( J^{-1}\nabla _{X}JY,Z\right) \\
&=&g\left( \nabla _{X}JY,Z\right) \\
&=&Xg\left( JY,Z\right) -g\left( JY,\nabla _{X}^{\ast \prime }Z\right) +\tau
\left( X\right) g\left( JY,Z\right) \\
&=&XG\left( Y,Z\right) -G\left( Y,\nabla _{X}^{\ast \prime }Z\right) +\tau
\left( X\right) G\left( Y,Z\right) .
\end{eqnarray*}
\end{proof}

Now, we consider another generalized connection which is called
dual-projectively equivalent connection.

\begin{definition}
\cite{Zhang} Let $(M,J,g)$ be a metallic pseudo-Riemannian manifold and $%
\nabla $ and $\nabla ^{\prime }$ be linear connections on $M$. We say that $%
\nabla $ and $\nabla ^{\prime }$ are dual-projectively equivalent if there
exists a $1$-form $\tau $ such that%
\begin{equation*}
\nabla _{X}^{\prime }Y=\nabla _{X}Y-g\left( X,Y\right) \tau ^{\sharp },
\end{equation*}%
where $g\left( X,\tau ^{\sharp }\right) =\tau \left( X\right) $.
\end{definition}

The following proposition ends this section.

\begin{proposition}
If $\left( \nabla ,J\right) $ is Codazzi-coupled, then each of the pairs $%
(\nabla ^{\prime },J),$ $(\nabla ^{\prime },J^{-1}),$ $(\left( \nabla
^{J}\right) ^{\prime },J)$ and $(\left( \nabla ^{J}\right) ^{\prime
},J^{-1}) $ is Codazzi-coupled.
\end{proposition}

\begin{proof}
Assume that the pair $\left( \nabla ,J\right) $ is Codazzi-coupled. Then we
have%
\begin{eqnarray*}
\left( \nabla _{X}^{\prime }J\right) Y-\left( \nabla _{Y}^{\prime }J\right)
X &=&\nabla _{X}^{\prime }JY-J\nabla _{X}^{\prime }Y-\nabla _{Y}^{\prime
}JX+J\nabla _{Y}^{\prime }X \\
&=&\nabla _{X}JY-g\left( X,JY\right) \tau ^{\sharp }-\nabla _{Y}JX+g\left(
Y,JX\right) \tau ^{\sharp } \\
&&-J\left( \nabla _{X}Y-g\left( X,Y\right) \tau ^{\sharp }-\nabla
_{Y}X+g\left( Y,X\right) \tau ^{\sharp }\right) \\
&=&\left( \nabla _{X}J\right) Y-\left( \nabla _{Y}J\right) X \\
&=&0.
\end{eqnarray*}%
Similarly, it can be easily shown that the pairs $(\nabla ^{\prime },J^{-1})$
$,$ $(\left( \nabla ^{J}\right) ^{\prime },J)$ and $(\left( \nabla
^{J}\right) ^{\prime },J^{-1})$ are Codazzi-coupled too.
\end{proof}

\section{Meatallic-like pseudo-Riemannian manifolds}

In this section we shall present some results concerning with a
metallic-like pseudo-Riemannian manifold.

\begin{definition}
\cite{Takano}\label{Defi1} Let $\left( M,g\right) $ be a pseudo-Riemannian
manifold with a metallic structure $J$ which has another $\left( 1,1\right)
- $tensor field $J^{\ast }$ satisfying%
\begin{equation*}
g\left( JX,Y\right) =g\left( X,J^{\ast }Y\right) ,
\end{equation*}%
where $\left( J^{\ast }\right) ^{\ast }=J$. Then $(M,J,g)$ is called a
metallic-like pseudo-Riemannian manifold.
\end{definition}

Standard calculations immediately give the following proposition.

\begin{proposition}
\label{Proposition8}Let $(M,J,g)$ be a metallic-like pseudo-Riemannian
manifold. Then, the following properties hold

$1)\ \left( J^{\ast }\right) ^{2}=pJ^{\ast }+qI,$

$2)\ \left( J^{\ast }\right) ^{-1}=\frac{1}{q}J^{\ast }-\frac{p}{q}I,$

$3)\ g\left( J^{-1}X,Y\right) =g\left( X,\left( J^{\ast }\right)
^{-1}Y\right) ,$

$4)\ g\left( JX,J^{\ast }Y\right) =pg\left( X,J^{\ast }Y\right) +qg\left(
X,Y\right) ,$

$5)\ g\left( \left( \nabla _{X}J^{\ast }\right) Y,Z\right) =g\left( Y,\left(
\nabla _{X}^{\ast }J\right) Z\right) ,$

$6)\ g\left( \left( \nabla _{X}^{\ast }J^{\ast }\right) Y,Z\right) =g\left(
Y,\left( \nabla _{X}J\right) Z\right) .$
\end{proposition}

\begin{proposition}
\label{Proposition9} Let $(M,J,g)$ be a metallic-like pseudo-Riemannian
manifold. Then we have

$1)\ \nabla $ is a $J^{\ast }-$connection if and only if $\nabla ^{\ast }$
is a $J-$connection.

$2)\ \nabla $ is a $J-$connection if and only if $\nabla ^{\ast }$ is a $%
J^{\ast }-$connection.
\end{proposition}

\begin{proof}
$1)\ $Assume that $\nabla J^{\ast }=0$. From (\ref{GD1}), we can write%
\begin{eqnarray*}
g\left( Y,\nabla _{X}^{\ast }JZ\right) &=&Xg\left( Y,JZ\right) -g\left(
\nabla _{X}Y,JZ\right) \\
&=&Xg\left( J^{\ast }Y,Z\right) -g\left( \nabla _{X}J^{\ast }Y,Z\right) \\
&=&g\left( J^{\ast }Y,\nabla _{X}^{\ast }Z\right) \\
&=&g\left( Y,J\nabla _{X}^{\ast }Z\right)
\end{eqnarray*}%
such that $\nabla ^{\ast }JZ=J\nabla _{X}^{\ast }Z$. Conversely, if $\nabla
^{\ast }$ is $J-$connection, then 
\begin{eqnarray*}
g\left( Y,\nabla _{X}J^{\ast }Z\right) &=&Xg\left( J^{\ast }Z,Y\right)
-g\left( J^{\ast }Z,\nabla _{X}^{\ast }Y\right) \\
&=&Xg\left( Y,JY\right) -g\left( Z,\nabla ^{\ast }JY\right) \\
&=&g\left( \nabla _{X}Z,JY\right) \\
&=&g\left( Y,J^{\ast }\nabla _{X}Z\right)
\end{eqnarray*}%
which implies $\nabla _{X}J^{\ast }Z=J^{\ast }\nabla _{X}Z$.

$2)$ Similarly, it can be obtained with standard calculations.
\end{proof}

\begin{proposition}
Let $(M,J,g)$ be a metallic-like pseudo-Riemannian manifold.

$1)$ If both $\left( \nabla ,J^{\ast }\right) $ and $\left( \nabla ^{\ast
},J\right) $ are Codazzi-coupled, then $\nabla J^{\ast }=\nabla ^{\ast }J.$

$2)\ $If both $\left( \nabla ^{\ast },J^{\ast }\right) $ and $\left( \nabla
,J\right) $ are Codazzi-coupled, then $\nabla ^{\ast }J^{\ast }=\nabla J.$
\end{proposition}

\begin{proof}
$1)$ From $(5)$ of Proposition \ref{Proposition8}, we have%
\begin{eqnarray*}
0 &=&g\left( \left( \nabla _{X}J^{\ast }\right) Y-\left( \nabla _{Y}J^{\ast
}\right) X,Z\right) \\
&=&g\left( Y,\left( \nabla _{X}^{\ast }J\right) Z\right) -g\left( X,\left(
\nabla _{Y}^{\ast }J\right) Z\right) \\
&=&g\left( Y,\left( \nabla _{Z}^{\ast }J\right) X\right) -g\left( X,\left(
\nabla _{Z}^{\ast }J\right) Y\right) \\
&=&g\left( X,\left( \nabla _{Z}J^{\ast }\right) Y\right) -g\left( X,\left(
\nabla _{Z}^{\ast }J\right) Y\right) \\
&=&g\left( Y,\left( \nabla _{Z}J^{\ast }\right) Y-\left( \nabla _{Z}^{\ast
}J\right) Y\right)
\end{eqnarray*}%
which implies $\nabla J^{\ast }=\nabla ^{\ast }J$.

$2)$ Similarly, it can be easily obtained.
\end{proof}

\begin{proposition}
\label{Proposition10}On a metallic-like pseudo-Riemannian manifold $(M,J,g)$%
, the following two identities hold%
\begin{equation*}
C\left( JX,Y,Z\right) =C\left( X,J^{\ast }Y,Z\right) +g\left( Y,\left(
\nabla _{Z}^{\ast }J\right) X-\left( \nabla _{Z}J\right) X\right) ,
\end{equation*}%
where $C=\nabla g$, and%
\begin{equation*}
C^{\ast }\left( JX,Y,Z\right) =C^{\ast }\left( X,J^{\ast }Y,Z\right)
+g\left( X,\left( \nabla _{Z}^{\ast }J^{\ast }\right) Y-\left( \nabla
_{Z}J^{\ast }\right) Y\right) ,
\end{equation*}%
where $C^{\ast }=\nabla ^{\ast }g.$
\end{proposition}

\begin{proof}
From Definition \ref{Defi1}, we get%
\begin{eqnarray*}
C\left( JX,Y,Z\right) &=&Zg\left( JX,Y\right) -g\left( \nabla
_{Z}JX,Y\right) -g\left( JX,\nabla _{Z}Y\right) \\
&=&Zg\left( X,J^{\ast }Y\right) -g\left( \nabla _{Z}X,J^{\ast }Y\right)
-g\left( X,\nabla _{Z}J^{\ast }Y\right) \\
&&+g\left( X,\left( \nabla _{Z}J^{\ast }\right) Y\right) -g\left( \left(
\nabla _{Z}J\right) X,Y\right) \\
&=&C\left( X,J^{\ast }Y,Z\right) +g\left( Y,\left( \nabla _{Z}^{\ast
}J\right) X-\left( \nabla _{Z}J\right) X\right)
\end{eqnarray*}%
and similarly%
\begin{equation*}
C^{\ast }\left( JX,Y,Z\right) =C^{\ast }\left( X,J^{\ast }Y,Z\right)
+g\left( X,\left( \nabla _{Z}^{\ast }J^{\ast }\right) Y-\left( \nabla
_{Z}J^{\ast }\right) Y\right) .
\end{equation*}
\end{proof}

Also note that if $\nabla J=0$ and $\nabla J^{\ast }=0,$ then we obtain 
\begin{equation*}
C\left( JX,Y,Z\right) =C\left( X,J^{\ast }Y,Z\right) =-C^{\ast }\left(
JX,Y,Z\right) =-C^{\ast }\left( X,J^{\ast }Y,Z\right) \text{.}
\end{equation*}

\begin{theorem}
On a metallic-like pseudo-Riemannian manifold $(M,J,g)$, if $\nabla J=0$, $%
\nabla J^{\ast }=0$ and $\left( \nabla ,g\right) $ is Codazzi-coupled, where 
$\nabla $ is a torsion-free connection, then we get%
\begin{equation*}
\left( \Phi _{J^{\ast }}g\right) \left( X,Y,Z\right) +\left( \Phi
_{J}g\right) \left( X,Y,Z\right) =0.
\end{equation*}
\end{theorem}

\begin{proof}
Using the definition of the operator $\Phi _{J}g,$ we have%
\begin{eqnarray*}
\left( \Phi _{J^{\ast }}g\right) \left( X,Y,Z\right) ) &=&J^{\ast }Xg\left(
Y,Z\right) -g\left( \nabla _{J^{\ast }X}Y,Z\right) -g\left( Y,\nabla
_{J^{\ast }X}Z\right) \\
&&-\left( Xg\left( J^{\ast }Y,Z\right) -g\left( J^{\ast }\nabla
_{X}Y,Z\right) -g\left( Y,J^{\ast }\nabla _{X}Z\right) \right) \\
&&+g\left( \left( \nabla _{Y}J^{\ast }\right) X,Z\right) +g\left( Y,\left(
\nabla _{Z}J^{\ast }\right) X\right) \\
&=&C\left( Y,Z,J^{\ast }X\right) -C\left( Y,JZ,X\right) -g\left( Y,\nabla
_{X}JZ\right) +g\left( Y,J^{\ast }\nabla _{X}Z\right) .
\end{eqnarray*}%
From the hypothesis and Proposition \ref{Proposition10}, we get%
\begin{equation*}
\left( \Phi _{J^{\ast }}g\right) \left( X,Y,Z\right) )=g\left( Y,\left(
J^{\ast }-J\right) \nabla _{X}Z\right) .
\end{equation*}%
Similarly, we can calculate%
\begin{equation*}
\left( \Phi _{J}g\right) \left( X,Y,Z\right) =g\left( Y,\left( J-J^{\ast
}\right) \nabla _{X}Z\right) .
\end{equation*}%
When we add the last two equations, we find 
\begin{equation*}
\left( \Phi _{J^{\ast }}g\right) \left( X,Y,Z\right) +\left( \Phi
_{J}g\right) \left( X,Y,Z\right) =0.
\end{equation*}
\end{proof}


\begin{thebibliography}{99}
\bibitem{Alek} D. V. Alekseevsky, V. Cortes and C. Devchand, Special complex
manifolds, J. Geom. Phys. 42(1--2) (2002), 85--105.

\bibitem{Bejan} C. L. Bejan, M. Crasmareanu, Conjugate connections with
respect to a quadratic endomorphism and duality, Filomat 30 (9) (2016),
2367--2374.

\bibitem{Blaga1} M. A. Blaga, A. Nannicini, On curvature tensors of Norden
and metallic pseudo-Riemannian manifolds, Complex Manifolds 6 (1) (2019),
150--159.

\bibitem{Blaga2} A. M. Blaga, M. Crasmareanu, The geometry of complex
conjugate connections, Hacet. J. Math. Stat. 41 (1) (2012), 119--126.

\bibitem{Zhang} O. Calin, H. Matsuzoe, J. Zhang, Generalizations of
conjugate connections. Trends in differential geometry, complex analysis and
mathematical physics, 26--34, World Sci. Publ., Hackensack, NJ, 2009.

\bibitem{Spinadel} V. W. de Spinadel, The metallic means family and
renormalization group techniques, Proc Steklov Inst Math Control Dynamic
Systems 1 (2000), 194--209.

\bibitem{Fei} T. Fei, J. Zhang, Interaction of Codazzi couplings with
(para-)K\"{a}hler geometry, Results Math. 72 (4) (2017), 2037--2056.

\bibitem{Gezer1} A. Gezer, H. Cakicioglu, Notes concerning Codazzi pairs on
almost anti-Hermitian manifolds, \ to appear Appl. Math. J. Chinese Univ.
(2022).

\bibitem{Goldberg} S. I. Goldberg, K. Yano, Polynomial structures on
manifolds, Kodai Math. Sem. Rep. 22 (1970), 199--218.

\bibitem{Salimov5} M. Iscan, A. A. Salimov, On K\"{a}hler-Norden manifolds,
Proc. Indian Acad. Sci. Math. Sci. 119 (1) (2009), 71--80.

\bibitem{Takano} K. Takano,Statistical manifolds with almost complex
structures, Tensor (N.S.) 72 (3) (2010), 225--231.

\bibitem{Lauritzen} S. Lauritzen. In: Statistical Manifolds, Eds. S. Amari,
O. Barndorff-Nielsen, R. Kass, S. Lauritzen, C. R. Rao, Differential
Geometry in Statistical Inference, IMS Lecture Notes, vol. 10, Institute of
Mathematical Statistics, Hayward, 1987, 163--216.

\bibitem{Manev} M. Manev, D. Mekerov, On Lie groups as quasi-K\"{a}hler
manifolds with Killing Norden metric, Adv. Geom. 8 (3) (2008), 343--352.

\bibitem{Matsuzoe} H. Matsuzoe, Statistical manifolds and its
generalization, in Proceedings of the 8th International Workshop on Complex
Structures and Vector Fields, World Scientific, 2007.

\bibitem{Nagaoka} H. Nagaoka, S. Amari, Differential geometry of smooth
families of probability distributions, Technical Report (METR) 82--7, Dept.
of Math. Eng. and Instr., Univ. of Tokyo, 1982.

\bibitem{Nomizu} K. Nomizu, Affine connections and their use, in Geometry
and Topology of Submanifolds VII , ed. F. Dillen, World Scientific, 1995.

\bibitem{Norden} A. P. Norden, Affinely Connected Spaces GRMFL, Moscow, 1976
(in Russian).

\bibitem{Vanzura} J. Vanzura, Integrability conditions for polynomial
structures, Kodai Math. Sem. Rep. 27 (1-2) (1976), 42-50.

\bibitem{Salimov1} A. A. Salimov, M. Iscan, F. Etayo, Paraholomorphic
B-manifold and its properties, Topology Appl. 154 (2007), 925-933.

\bibitem{Salimov2} A. Salimov, Tensor operators and their applications,
Mathematics Research Developments, Nova Science Publishers, Inc., New York,
2013. xii+186 pp.

\bibitem{Salimov3} A. A. Salimov, Almost analyticity of a Riemannian metric
and integrability of a structure, Trudy Geom. Sem. Kazan. Univ. 15 (1983),
72-78.

\bibitem{Salimov4} Salimov, A.A. A new method in the theory of liftings of
tensor fields in a tensor bundle; translation in Russian Math. (Iz. VUZ) 38
(3) (1994), 67-73.

\bibitem{Salimov6} A. A. Salimov, M. Iscan, K. Akbulut, Notes on
para-Norden-Walker 4-manifolds, Int. J. Geom. Methods Mod. Phys. 7 (8)
(2010), 1331--1347.

\bibitem{Tachiban} S. Tachibana, Analytic tensor and its generalization,
Tohoku Math. J. 12 (1960), 208-221.
\end{thebibliography}
\end{document}